\documentclass{article}

\pdfoutput=1

\usepackage{amsfonts,amssymb,amsmath,amsthm}

\usepackage{graphicx}
\usepackage{geometry}
\usepackage{xcolor}

\usepackage[colorlinks,
			pdffitwindow=true,
      			plainpages=false,
      			pdfpagelabels=true,
     			pdfpagemode=UseOutlines,
      			pdfpagelayout=SinglePage,
			bookmarks=false,
			colorlinks=true,
      			hyperfootnotes=false,
			linkcolor=blue,
			citecolor=green!50!black]{hyperref}

\usepackage{enumerate}

\newtheorem{definition}{Definition}

\newtheorem{lemma}[definition]{Lemma}
\newtheorem{corollary}[definition]{Corollary}

\newtheorem{remark}[definition]{Remark}
\newtheorem{proposition}[definition]{Proposition}

\newcommand{\R}{\mathbb{R}}

\newcommand{\N}{\mathbb{N}}

\newcommand{\set}[1]{\mathcal{#1}}
\newcommand{\mat}[1]{\mathbf{#1}}

\def \prob {\mathbb{P}}

\newcommand{\ep}{\varepsilon}

\title{Optimizing the stable behavior of parameter-dependent dynamical systems --- maximal domains of attraction, minimal absorption times}
\author{P\'eter~Koltai\thanks{M3, Faculty for Mathematics, Technische Universit\"at M\"unchen, Boltzmannstr.\ 3, 85748 Garching. E-mail: koltai{@}ma.tum.de} \and Alexander~Volf\thanks{Research Group CAPS - Computer Aided Plastic Surgery, Clinic for Plastic Surgery and Hand Surgery, Klinikum rechts der Isar, Technische Universit\"at M\"unchen, Ismaninger Str. 22, D-81675 M\"unchen, Germany.  E-mail: \newline alexander.volf{@}mytum.de} \thanks{CADFEM GmbH, Marktplatz 2, D-85567 Grafing b. M\"unchen, Germany. E-mail: avolf{@}cadfem.de}}
\date{}

\begin{document}

\maketitle

\begin{abstract}
We propose a method for approximating solutions to optimization problems involving the \textit{global} stability properties of parameter-dependent continuous-time autonomous dynamical systems. The method relies on an approximation of the infinite-state deterministic system by a finite-state non-deterministic one --- a Markov jump process. The key properties of the method are that it does not use any trajectory simulation, and that the parameters and objective function are in a simple (and except for a system of linear equations) explicit relationship.
\end{abstract}

\section{Introduction}

\paragraph{The problem.}
In numerous applications one is faced with the problem of steering a continuous time dynamical system which depends on some fixed (control) parameters into a target region. The word \textit{fixed} refers to the case where these parameters can be set at the beginning, but they cannot be changed once the system is ``running''.

Not only the asymptotic behavior (i.e.\ entering the target region) is of interest. Dealing with the nonlinear behavior of the system outside the target region is also important, it is desirable to achieve some kind of optimal stability properties; e.g.\ the domain of attraction\footnote{Also called the \textit{basin of attraction}, the \textit{region of asymptotic stability}, etc.}~(DOA) of the target region should be maximal, or the average time the system needs to reach the region should be minimal.

The literature extensively studies domains of attraction and their numerical computation~(cf.\ below), there are even approaches aiming at the \textit{local} enlargement of these objects~\cite{DaKu71, ShSt75, Liaw06}. It seems however, that little to no effort has been done in considering this task as an optimization problem, stated above. The current work proposes a method tackling this particular problem, and acting thereby \textit{globally}.

\paragraph{Possible tools and previous work.}

Numerical methods computing the DOA of a particular system can be roughly divided into three major groups.
\begin{itemize}
	\item The first are methods using direct simulation (forward or backward) of the underlying ordinary differential equation~\cite{GeTaVi85, Hsu80, HsGu80, FlGu88, Gru01}, in order to obtain a set covering the DOA, or some approximation on its boundary.
	\item The second are Lyapunov function based techniques, e.g.~\cite{LaLe61,Gie09} or Zubov's method~\cite{Zub64} (see also in~\cite{Hahn67}). Usually they allow one to extract a subset of the DOA as a level set of a function solving some partial differential equation, and no trajectories have to be simulated.
	\item The third group consists of probabilistic approaches, with trajectory simulation~\cite{Gol04}, or without~\cite{Kol10}. Absorption probabilities for an approximate non-deterministic system reveal properties of the DOA. Probabilistic approaches for solving stochastic control problems have been studied by Kushner et al.~\cite{Kush77,KuDu01}.
\end{itemize}
In order to view the stability behavior of the system as an optimization problem, we choose a simple setting: let a real valued objective function be given, which depends on the stability properties of the system in a desired way; such a function could be e.g.\ the Lebesgue volume of the DOA. To make the optimization more convenient to carry out, we assume the dependence of this objective function on the control parameter to be sufficiently smooth (at least differentiable), that we can use a simple gradient method to get to a (local) maximum/minimum.

Along the lines of the above example (maximization of the volume of the DOA) we would like to highlight some difficulties arising if we try to use one of the above methods for optimization by the gradient method.\\
   For methods using direct simulation, computing the gradient of the objective function will require by the chain rule the computation of derivatives of the flow with respect to (w.r.t.) the parameters. Thus, variational equations for the underlying ordinary differential equation (ODE) have to be solved. These are computationally expensive, since they are in general higher dimensional than the original ODE\footnote{They are, in fact, an ODE for the so-called \textit{sensitivity matrix w.r.t.\ the parameters}. The dimension of the sensitivity matrix is the state space dimension times the parameter space dimension.}, and they also require the derivative of the vector field (which also has to be computed numerically).\\
   The Lyapunov function based techniques deliver only the boundary of the DOA as level set of some scalar function. It is not clear how to obtain the derivative of the level set w.r.t.\ the parameters, and if there is a way to devise an efficient method based on this. Further, the most Lyapunov function based techniques deliver merely a subset of the DOA which' topology depends on the characteristics of the chosen Lyapunov function class (e.g., the level sets of quadratic functions are always ellipsoids).

The considerations above should give the reader a first insight into what one has to deal with when trying to solve an optimization problem associated with the DOA. Nevertheless, the main purpose of this work is not the comparison of the different approaches one could use for optimizing the objective function, rather to analyze and show the feasibility of the method we chose.


%
%
%

\paragraph{The approach.}

Our method is based on a set-oriented approach to approximate properties of dynamical systems~\cite{DeHo97, DeJu98}. We partition the state space~$\set{X}$ into disjoint sets, which serve to discretize the original deterministic dynamical system, and relate it to a finite state non-deterministic system --- a Markov jump process. Then we construct the infinitesimal generator (matrix) of the latter by integrating the vector field of the ODE on the boundary of the partition elements. From this matrix we compute absorption probabilities for the finite state system, which yield the desired information about the DOA; or we compute expected absorption times, which serve as approximations of the absorption times of the original system. For example, the volume of the DOA is simply the integral over the absorption probability function associated with the original deterministic system.\footnote{Since this is 1 if a given point is contained in the DOA, and 0 otherwise.}

All of the involved computations are explicit (except for the solution of a system of linear equations), and it is easy to obtain their derivatives. Composing these derivatives according to the chain rule yields the derivative of the objective function w.r.t.\ the parameters, and allows the application of local optimization methods, e.g.\ the gradient method.

\paragraph{Outline.}
This work is structured as follows. Section~\ref{sec:background} collects the mathematical background on deterministic and non-deterministic dynamical systems which are going to be used afterwards. In Section~\ref{sec:discretization} we introduce the discretization we use to make a finite-state non-deterministic system from the original infinite-state deterministic one. Section~\ref{sec:altMeasures} is an outlook: it collects possible alternatives for the approximation of the absorption times of the deterministic system by quantities derived from the non-deterministic one. Section~\ref{sec:application} is the main part of this work stating the optimization problem and showing how we propose to solve it. Numerical examples follow in Section~\ref{sec:numexp}, and some concluding remarks are collected in Section~\ref{sec:conclusions}.

\section{Background}	\label{sec:background}

\subsection{Dynamical systems}	\label{ssec:DS}

The temporal variation of states $x(t)\in\R^d$ we are interested in is given by an (autonomous) ODE ${\dot x = v(x)}$, where ${v:\R^d\to\R^d}$, and ${\dot x := \tfrac{dx}{dt}}$ denotes the temporal derivative. Let $\phi^t$ denote the associated flow, i.e.\ ${x(t) = \phi^tx(0)}$. The set $\{\phi^tx\}_{t\ge0}$ is called the (forward) trajectory (of $x$).

We are only looking at a compact subset $\set{X}$ of the whole state space $\R^d$. If a trajectory leaves $\set{X}$ it immediately terminates; in this case we write $\phi^t x = \omega$, where $\omega$ is a fictive state representing everything outside $\set{X}$.\\
We call the pair $(\set{X},\phi^t)$ a dynamical system, and we refer to $\set{X}$ as state space. Later, we will focus on systems for which some states shall be steered into a \textit{target region} $\set{T}\subset \set{X}$ in a prescribed way. Following objects will be at the center of our attention.
\begin{definition}
For the target region $\set{T}\subset \set{X}$ we call the set
\[
\set{D}:=\left\{x\in \set{X}\,\big\vert\, \phi^tx\in \set{T}\text{ for a }t\ge0\text{ and }\phi^sx\in \set{X}\ \forall s\in[0,t]\right\}
\]
the \textit{domain of attraction} of $\set{T}$.\\
The quantity 
\[
\tau(x) := \left\{\begin{array}{ll}
									\inf\left\{t\ge0\,\big\vert\,\phi^tx\in \set{T}\right\}, & x\in \set{D},\\
									\infty, & \text{elsewise}.	
							 \end{array}\right.
\]
is called the \textit{absorption time} of $x$.
\end{definition}
Note that we are only interested in trajectories which stay all the time in~$\set{X}$, until being absorbed in~$\set{T}$.

\subsection{Markov chains}	\label{ssec:Markov}

We use the theory of Markov chains to describe non-deterministic dynamical behavior. From now on in this section let $\set{Y}=\{1,2,\ldots,n\}$ be a discrete state space. We also work with a given probability space $(\Omega,\mathcal{F},\prob)$.
\begin{definition}[Stochastic process]
Let $\mathcal{I} = \N$ or $\mathcal{I}=\R_{\ge0}$, and let $\set{Y}$ be the set of possible states. Then a family $\{Z^t\}_{t\in\mathcal{I}}$, where $Z^t:\Omega\to \set{Y}$ for all $t\in\mathcal{I}$, of $\set{Y}$-valued random variables is called a stochastic process.
\end{definition}
\begin{definition}[Discrete time Markov chain]	\label{D:discrMC}
We call a stochastic process $\{Z^t\}_{t\in \set{T}}$ a \textit{discrete time Markov chain}, if $\mathcal{I}=\N$, and $\mat{P}\in\R^{n\times n}$ with (i) $\mat{P}_{ij}\ge0$ and (ii) $\sum_{i=1}^n\mat{P}_{ij}\le 1$, describes the transition probabilities, i.e.\
\[
\mat{P}_{ji} = \prob\left(Z^{t+1}=j\,\big\vert\, Z^t=i\right) := \prob\left(\left\{Z^{t+1}=j\text{ provided }Z^t=i\right\}\right),
\]
for all $t\in\mathcal{I}$ and $i,j\in \set{Y}$. The matrix $\mat{P}$ is called the transition matrix. A matrix satisfying (i) and (ii) is called sub-stochastic. If (ii) holds with equality, $\mat{P}$ is called stochastic.
\end{definition}
\begin{definition}[Continuous time Markov chain]	 \label{D:contMC}
Let $\mathcal{I}=\R_{\ge0}$. Further, let $\mat{G}\in\R^{n\times n}$ be such that
\begin{itemize}
	\item[(i)] $\mat{G}_{ij}\ge 0$ for $i\neq j$, and
	\item[(ii)] $\sum_{i=1}^n \mat{G}_{ij}\le 0$.
\end{itemize}
Define $\mat{P}^t = e^{t\mat{G}} = \sum_{k=0}^{\infty}\tfrac{t^k\mat{G}^k}{k!}$. Then $\mat{P}^t$ is a sub-stochastic matrix for every $t\ge 0$.\footnote{For a proof, see Theorem~2.1.2 in~\cite{Nor97}.} A process $\{Z^t\}_{t\in\mathcal{I}}$ with
\[
\mat{P}^t_{ji} = \prob(Z^{t+s}=j\,\vert\, Z^s=i)\quad\text{for all }t,s\ge 0,\text{ and }i,j\in \set{Y}
\]
is called a \textit{continuous time Markov chain}. It is also often called a Markov jump process (MJP). The matrix $\mat{G}$ is called the (infinitesimal) \textit{generator} of the process.
\end{definition}

One may think of such MJPs as follows (cf.\ Section~2.6 in~\cite{Nor97}). Being in an arbitrary state $i$ at an arbitrary time $s$, the process remains at the given state until some random time $s+t$, when it jumps, independently of $t$, at random to another state $j$. The jump time $t$ has exponential distribution with parameter $\mat{G}_{ii}$, and the probability of jumping to state $j$ ($j\neq i$) is $-\mat{G}_{ji}/\mat{G}_{ii}$, unless $\mat{G}_{ii}=0$, in which case the process does not leave the state $i$ ever.

Consider now a discrete time Markov chain with transition matrix $\mat{P}$. If $\mat{P}$ is sub-stochastic but not stochastic, there is an $i\in \set{Y}$ such that $\sum_{j\in \set{Y}}\mat{P}_{ji}=p_i<1$. In other words, if the process is currently in state $i$, there is a positive probability $1-p_i$ that it will not end up in $\set{Y}$ in the next step --- the process terminates. Analogous considerations can be done with a continuous time process. Its generator will satisfy $\sum_{j\in \set{Y}}\mat{G}_{ji}<0$ for some $i\in \set{Y}$. We call such processes \textit{leaky}. They reflect the restriction of our interest to the state space: everything leaving the state space is considered to be lost.

In the following we denote a Markov chain already terminated at time $t$ by $Z^t=\omega$. Here $\omega$ represents ``everything outside of the state space'' --- just as in the deterministic setting, Section~\ref{ssec:DS}.

\subsection{Absorption probabilities and expected absorption times}	\label{ssec:absProb}

Throughout this section we consider MJPs. Furthermore, we assume that the set $\set{T}\subset \set{Y}$ is absorbing (i.e.\ ${\prob(Z^t\in \set{T}\,\vert\,Z^0\in \set{T})=1}$), and that the process is possibly leaky. Then, under a reachability assumption, the process will either end in the absorbing set $\set{T}$, or ``leak out'' (i.e.\ terminate in the fictive state $\omega$; cf.\ above).\footnote{This reachability assumption is that there is a finite time $t>0$ such that $\prob\left(Z^t\in \set{T}\cup\{\omega\}\right)>0$ for any starting state. It implies that all the states $\set{Y}\setminus \set{T}$ are transient; see~\cite{Nor97} Theorem~3.4.2 p.115.} Let $p\in[0,1]^n$ denote the \textit{absorption probabilities} for the absorbing set, i.e.
\[
p_i = \prob\left(\left\{Z^t \in \set{T} \text{ for some }t\ge 0, \text{ provided }Z^0=i\right\}\right).
\]
The corresponding absorption time
\[
A = \inf\left\{t\,\big\vert\, Z^t\in\set{T}\right\}
\]
is a random variable itself; here we are interested in its expectation. The \textit{expected absorption times} are defined as
\[
a_i = \mathbb{E}_i(A) := \mathbb{E}\left(A\,\big\vert\,Z^0=i\right).
\]
An important quantity will be the \textit{expected termination time}. The termination time
\[
T = \inf\left\{t\,\big\vert\, Z^t\in\set{T}\cup\{\omega\}\right\}
\]
is a random variable measuring the time to termination either in the state $\omega$ or in the absorbing set $\set{T}$. We define
\[
t_i = \mathbb{E}_i(T).
\]
We wish to compute $p$, $a$ and $t$ from the generator $\mat{G}$ of the process.

Set $\widehat{\mat{G}}:= \mat{G}_{\set{Y}\setminus \set{T}, \set{Y}\setminus \set{T}}$; i.e.\ $\widehat{\mat{G}}$ is the matrix with the jump rates from $\set{Y}\setminus \set{T}$ to $\set{Y}\setminus \set{T}$. The transiency of the set $\set{Y}\setminus \set{T}$ is equivalent with $\lim_{t\to\infty}e^{t\widehat{\mat{G}}}=0$.
\begin{proposition}[\cite{Nor97} Theorem 3.3.1]	\label{prop:absprob}
Assume $\lim_{t\to\infty}e^{t\widehat{\mat{G}}}=0$. Then the absorption probabilities are the unique solution of
\begin{equation}
\begin{array}{rcll}
	\sum_{j\in \set{Y}\setminus \set{T}} p_j\mat{G}_{ji} & = & -\sum_{j\in \set{T}}\mat{G}_{ji}, & i\in \set{Y}\setminus \set{T},\\
	p_i & = & 1, & i\in \set{T}.
\end{array}
\label{eq:absprob}
\end{equation}
\end{proposition}
The main difference between the expected absorption times and termination times can be seen without computation. It holds $a_{\omega}=\infty$, but $t_{\omega}=0$. A representation of these quantities involving the generator has to respect this fact.
\begin{proposition}[\cite{Nor97} Theorem 3.3.3]	\label{prop:termtimes}
Assume $\lim_{t\to\infty}e^{t\widehat{\mat{G}}}=0$. Then the expected termination times are the unique solution of
\begin{equation}
\begin{array}{rcll}
	\sum_{j\in \set{Y}\setminus \set{T}} t_j\mat{G}_{ji} & = & -1, & i\in \set{Y}\setminus \set{T},\\
	t_i & = & 0, & i\in \set{T}\cup\{\omega\}.
\end{array}
\label{eq:termtimes}
\end{equation}
\end{proposition}
Along the lines of the proof of Proposition~\ref{prop:termtimes} we obtain
\begin{corollary}	\label{cor:abstimes}
Assume $\lim_{t\to\infty}e^{t\widehat{\mat{G}}}=0$. Then the expected absorption times are the unique solution of
\begin{equation}
\begin{array}{rcll}
	\sum_{j\in \set{Y}\setminus \set{T}\cup\{\omega\}} a_j\mat{G}_{ji} & = & -1, & i\in \set{Y}\setminus \set{T},\\
	a_i & = & 0, & i\in \set{T},\\
	a_i & = & \infty, & i=\omega,
\end{array}
\label{eq:abstimes}
\end{equation}
where $\mat{G}_{\omega i}:= -\mat{G}_{ii}-\sum_{j\in \set{Y}}\mat{G}_{ji}$, the termination rate from state $i$.
\end{corollary}

As one can see, the quantities $p$, $a$ and $t$ can be computed by solving systems of linear equations.

There is a problem with $a$ which will pose us difficulties when computing absorption times from a discretization of the system. Starting at $i$, if there is a non-zero probability (does not matter how small) of terminating in $\omega$ at some future time, the expected absorption time will be infinite. Even if these non-zero probabilities come from discretization errors, it makes the resulting absorption times not applicable for the approximation of absorption times of the original system.\\
Note that $a_{\{p=1\}} = t_{\{p=1\}}$, i.e.\ $a$ and $t$ coincide for the states which have absorption probability one; clearly because these states terminate in~$\set{T}$.

Alternative measures of absorption times are introduced in Section~\ref{sec:altMeasures} below.

\section{Discretization}	\label{sec:discretization}

The desired objects of the deterministic system defined by $\dot x=v(x)$ on $\set{X}$ will be computed from the discretization given below.

Partition the state space $\set{X}$ into finitely many disjoint partition elements $\set{X}_1,\ldots,\set{X}_n$, where each set $\set{X}_i$ has a piecewise smooth boundary $\partial \set{X}_i$, such that the unit outer normal vector $n_j$ exists almost everywhere (measured by the $d-1$ dimensional Lebesgue measure $m_{d-1}$ on $\partial \set{X}_i$). Usually, the $\set{X}_i$ are rectangles or simplices.
\begin{definition}
The discrete generator matrix $\mat{G}_n\in\R^{n\times n}$ associated with the vector field $v$ is defined by
\begin{equation}
\renewcommand{\arraystretch}{1.2}
\mat{G}_{n,ij} = \left\{\begin{array}{ll}
										(1/m(\set{X}_j))\int_{\partial \set{X}_i\cap\partial \set{X}_j} \left(v(x)\cdot n_j(x)\right)^+\, dm_{d-1}(x), & i\neq j \\
										-(1/m(\set{X}_i))\int_{\partial \set{X}_i} \left(v(x)\cdot n_i(x)\right)^+\, dm_{d-1}(x), & i=j,
									\end{array}\right.
\label{eq:disc_gener}
\end{equation}
where $x\cdot y$ denotes the dot product of the vectors $x$ and $y$, and $f^+$ denotes the positive part of the function $f$.
\end{definition}
The discretization goes back to~\cite{FroJK10}, but apart from the probabilistic point of view, it is only the spatial discretization of the upwind method known from finite volume methods~\cite{Lev02}. A first application of this idea for the computation of the DOA appears in~\cite{Kol10}. It is shown in the latter work that $\mat{G}_n$ indeed generates a MJP on the state space with elements $\set{X}_1,\ldots,\set{X}_n$, this MJP can be associated with a non-deterministic system on $\set{X}$ which converges to $\phi^t$ in distribution as $n\to\infty$. This elucidates why we use absorption probabilities and expected absorption times of the MJP generated by $\mat{G}_n$ in order to approximate the DOA and absorption times of $\phi^t$.

Due to this convergence in distribution if one starts the MJP on a sufficiently fine discretization with a distribution highly concentrated around a state $x\in\set{D}$, the distribution of $Z_n^t$ is highly concentrated around $\phi^t x$. Now, if $t$ is chosen such that $\phi^t\in\set{T}$ (it is possible per definition of~$\set{T}$) then $Z_n^t$ is already absorbed with high probability. This tells us intuitively that for a sufficiently fine partition $p_i\approx 1$ if $\set{X}_i\subset\set{D}$; and similarly that the expected absorption times of~$Z_n^t$ approximate the absorption times of~$\phi^t$ well.\footnote{The reader may have already noticed that we are avoiding precise statements about the convergence of absorption probabilities and absorption/termination times. We decided to do so due to the fact that proofs require more advanced mathematical tools, which' introduction would highly complicate the presentation and distract the attention from the main purpose of this work. Precise statements and their proofs are subject of ongoing work, and will appear elsewhere.}

\begin{remark}
The consideration here should help to associate the system $\phi^t$ with the MJP generated by~$\mat{G}_n$. Let $Z$ be a random variable distributed uniformly in $\set{X}_i$. Then we have by Lemma~4.1 and Lemma~4.4 in~\cite{FroJK10}
\[
\frac{d}{dt}\prob\left(\phi^tZ\in \set{X}_j\right)\!\Big\vert_{t=0} = \mat{G}_{n,ji}.
\]
If now $\{Z_n^t\}_{t\ge0}$ is the MJP generated by $\mat{G}_n$, we also have from the definition that
\[
\frac{d}{dt}\prob\left(Z_n^t = \set{X}_j\,\big\vert\,Z_n^0 = \set{X}_i\right)\!\Big\vert_{t=0} = \mat{G}_{n,ji}.
\]
\end{remark}

Note that unless the partition elements $\set{X}_i$ and $\set{X}_j$ have a fully $d-1$ dimensional intersection $\partial \set{X}_i\cap\partial \set{X}_j$, it holds $\mat{G}_{n,ij}=0$. Thus, $\mat{G}_n$ is sparse.\\
Another numerical advantage of the discretization is that it does not use trajectory simulation --- only integrals of at least continuous functions (the $(v\cdot n_j)^+$) on $d-1$ dimensional domains have to be computed.

\noindent\begin{minipage}[b]{0.55\textwidth}
\begin{remark}[The ``wrapping effect'']
Having seen how the discretization works we can go back to the effect mentioned at the very end of Section~\ref{sec:background}. The figure to the right shows a trajectory of a stable system spiraling into one point, say, the origin. The stability is weak, since the decrease of the distance to the origin during one rotation is small. The rectangles in the figure indicate the elements~$\set{X}_i$ of the chosen partition. Computing the discrete generator~$\mat{G}_n$ of this system, one can see, that no matter in which partition element one starts in, there is always a positive probability to leak out of the state space. This effect decreases as the partition gets finer, it does not vanish completely however, so the absorption times~$a_n$ computed from~$\mat{G}_n$ will be infinity for all boxes except the target (per definition). To remedy this fact one can work with the termination times, or with other constructs; see Section~\ref{sec:altMeasures} below.
\end{remark}
\end{minipage}
\hfill
\begin{minipage}[b]{0.4\textwidth}
\begin{center}
\includegraphics[width=1.0\textwidth]{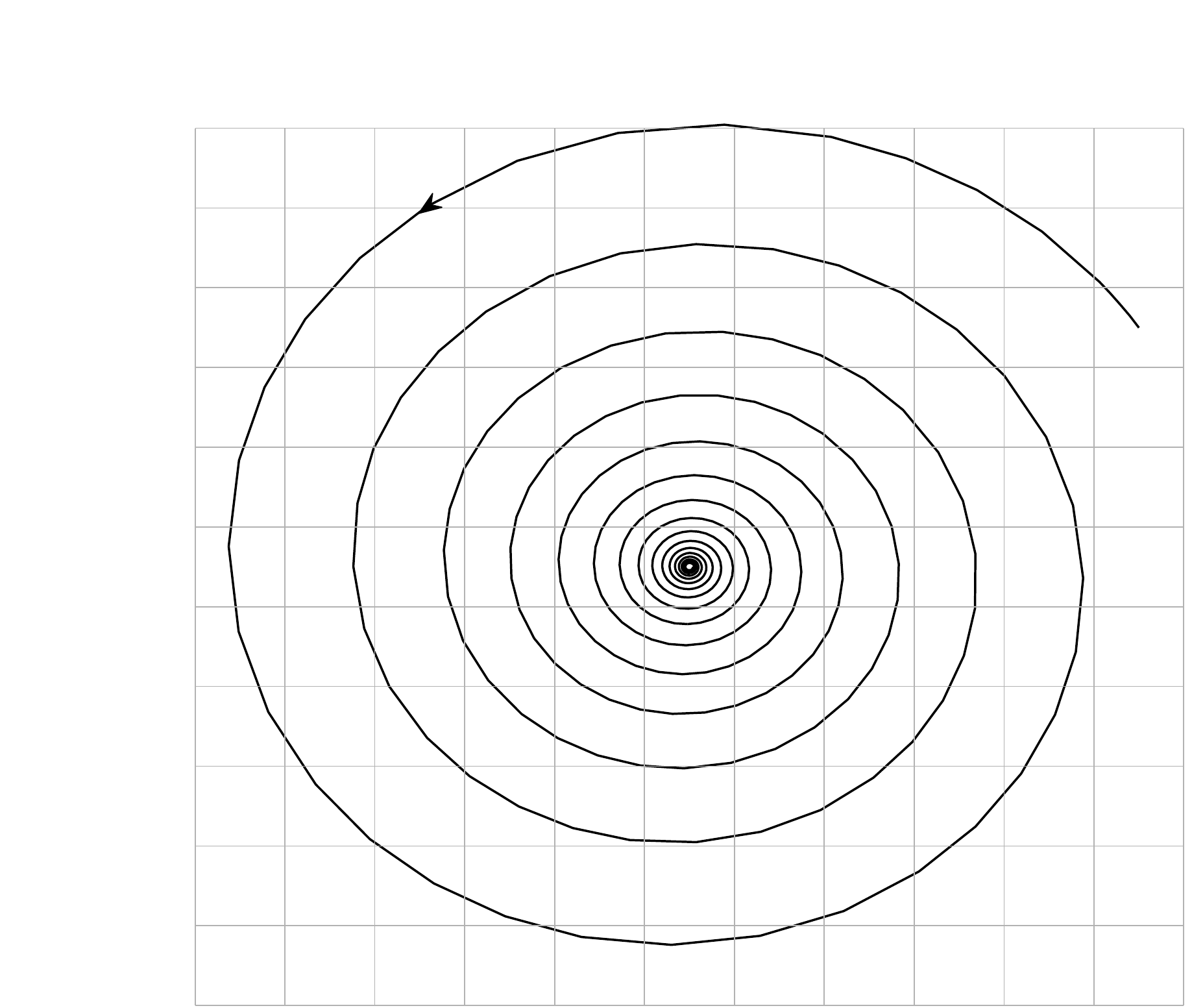}
\end{center}

\vspace{0.5pt}
The ``wrapping effect'': no matter how small the partition elements are, there is always a nonzero probability for the induced MJP to leak out.
\end{minipage}

\section{Alternative measures for the absorption time}	\label{sec:altMeasures}

It has been already discussed that the expected absorption times $a_n$ of the discrete generator~$\mat{G}_n$ are not the right tool to approximate the absorption times $\tau$ of the deterministic system. Also the approximation through the expected termination times $t_n$ may be highly defective; e.g.\ when the time to absorption (if absorption occurs) and the time to leak-out (if leaking out occurs) are way apart.

We would like to discuss here two alternative approximations to $\tau$ by some quantities derived from the discrete generator which only depend on absorption times; i.e.\ on the random variable~$A$. Both are solutions to a system of linear equations, thus their computation is not a harder task than that of the previously used quantities. Their incorporation into our approach presented below, and the analysis of their numerical properties are subject of future work.

\paragraph{Transformation.}

Let us recall that $A$ denotes the random variable of absorption times. Since \linebreak[4] ${\prob(A=\infty)>0}$, we define
\[
h_i = \mathbb{E}_i\left(e^{-A}\right).
\]
The advantage is that $h$ attains always finite values, even if the expected absorption time~$a$ doesn't. The transformation $\tau\mapsto 1-e^{-\tau}$ is called the Kru{\v z}kov transform. Note that $1-e^{-\tau}$ is continuous on the entire state space~$\set{X}$ if $\tau$ is continuous on $\set{D}$.

One can show that for a MJP with generator matrix $\mat{G}$ holds
\begin{equation}
\begin{array}{rcll}
	h_i & = & \sum_{j\in \set{Y}\setminus \set{T}} h_j\mat{G}_{ji} + \sum_{j\in\set{T}}\mat{G}_{ji}, & i\in \set{Y}\setminus \set{T},\\
	h_i & = & 1, & i\in \set{T}.
\end{array}
\label{eq:kruzkovtimes}
\end{equation}
If $h_n$ is the solution of~\eqref{eq:kruzkovtimes} corresponding to the discrete generator $\mat{G}_n$, we may use $-\log(h_n)$ to approximate $\tau$.

\paragraph{Absorption times conditioned to absorption.}

After finishing this work the following idea of P.\ Pollett and D.\ M.\ Walker has been pointed out to us.

If we expect a box to lie in the DOA but the corresponding absorption probability is smaller than~1, the expected absorption time conditioned to that absorption occurs in finite time~\cite{Walk98} is a very natural approximation to $\tau$. Hence, define
\[
a_i^* = \mathbb{E}_i\left(A\,\big\vert\, A<\infty\right)
\]
on
\[
\set{Y}^* = \left\{i\in\set{Y}\,\big\vert\, p_i>0\right\},
\]
where $p$ denotes the absorption probabilities. Then by Theorem~2~\cite{Walk98} we have
\begin{equation}
\begin{array}{rcll}
	\sum_{j\in \set{Y^*}\setminus \set{T}} a_j^*p_j\mat{G}_{ji} & = & -p_i, & i\in \set{Y}^*\setminus \set{T},\\
	a_i^* & = & 1, & i\in \set{T}.
\end{array}
\label{eq:condtimes}
\end{equation}
Since the magnitude of the $p_i$ may range over several orders of magnitude (in practice between machine precision to one), any numerical method using~\eqref{eq:condtimes} has to address stability issues.

\section{Application} \label{sec:application}

\subsection{Parameter dependent system and objective function}

The knowledge gained in the previous sections will be now applied to optimize the stable behavior of dynamical systems arising from the vector field $v(x;b)$, where $b\in\R^r$ is an adjustable parameter.

The quality of the dynamical behavior is going to be measured by an objective function $f:\R^r\to\R$ which will be subject to minimization or maximization w.r.t.\ $b$. This objective function will consist of two parts; one representing the desired behavior, and one representing the costs connected with the choice of the control parameter. Our two model problems are as follows.

\paragraph{Maximal DOA.}
The goal here is to obtain a subset of the state space with the biggest possible volume which is steered into the target region~$\set{T}$. Correspondingly, one would like to maximize the volume $m(\set{D}(b))$ of the set~$\set{D}(b)$, the DOA associated with~$v(\cdot;b)$. We define the objective function as
\begin{equation}
f(b) := m(\set{D}(b))-\alpha |b|^2,
\label{eq:maxDOA}
\end{equation}
with $\alpha>0$ being some penalty parameter, and $|b|^2:=\sum_{i=1}^rb_i^2$.

Observe that $m(\set{D}(b)) = \int_{\set{X}}\chi_{\set{D}(b)}(x)\,dx$, where $\chi_{\set{D}(b)}(x)$, the characteristic function of the set~$\set{D}(b)$, acts as a function of absorption probabilities for the deterministic system: if $x$ is steered into~$\set{T}$ then $\chi_{\set{D}(b)}(x)=1$, and 0 otherwise. Suppose a partition $\{\set{X}_1,\ldots,\set{X}_n\}$ satisfying the conditions of Section~\ref{sec:discretization} is given. By computing the absorption probabilities~$p_n$ from the discrete generator associated with $v(\cdot;b)$ for this partition, we obtain an approximation of the objective function:
\[
f(b) \approx f_n(b):= \sum_{i=1}^nm(\set{X}_i)p_{n,i} - \alpha|b|^2.
\]
Our strategy is to use this approximation in order to carry out the maximization.

\paragraph{Minimal absorption times.}
Here we have a prescribed set~$\set{D}_0\supset\set{T}$, and the goal is to minimize the average absorption time over this set. The objective function is defined as
\begin{equation}
f(b) := \int_{\set{D}_0}\tau(x;b)\,dx + \alpha|b|^2,
\label{eq:minTime}
\end{equation}
with some penalty parameter~$\alpha>0$; where $\tau(\cdot;b)$ is the function of absorption times associated with~$v(\cdot;b)$. The approximation of this objective function is going to be discussed in Section~\ref{ssec:optimTime} below.

\medskip
\begin{remark}[Other objective functions]	\label{rem:otherObjectives}
Depending on the needs of the application one can work with other objective functions as well. Here we present two of them which can be handled by our method. They are not going to be discussed further in the following.
\begin{enumerate}[(a)]
	\item If it is of higher importance that the DOA covers some specific regions, one could take some function $\theta:\set{X}\to\R$ to weight absorption probabili\-ties. The objective function
	\[
	f(b) = \int_{\set{D}(b)}\theta(x)\,dx - \alpha|b|^2
	\]
	has a clear representation in terms of absorption probabilities.
	\item Average absorption speed maximization over a domain~$\set{D}_0$ of interest:
	\[
	f(b) = \int_{\set{D}_0}\frac{|x|}{\tau(x)}\,dx - \alpha|b|^2.
	\]
\end{enumerate}
\end{remark}

\subsection{Iterative optimization}	\label{ssec:optim}

In this section we describe the optimization procedure for finding (approximate) local extrema of the objective function~$f$ by performing a local optimization of its approximation~$f_n$. For the demonstration we choose the problem of DOA maximization.

Suppose we have chosen a partition $\{\set{X}_1,\ldots,\set{X}_n\}$ of the state space $\set{X}$ such that the target region~$\set{T}$ is the union of some partition elements. Later on we will use the short-hand notation~$i\in\set{T}$ for~$\set{X}_i\subseteq\set{T}$. Set $f_n(b)=\sum_{i=1}^n m(\set{X}_i)p_{n,i} - \alpha|b|^2$.

Next we show that under fairly general conditions the objective function~$f_n$ is differentiable w.r.t.\ $b$. Then, by the gradient method
\[
b_{k+1} = b_k + Df_n(b_k),\qquad k=0,1,\ldots
\]
we can compute a local maximum of it. Here and in the following $Dg$ denotes the derivative of the function $g$. The value $f_n(b)$ is computed by the sequence of mappings
\[
b \mapsto v(\cdot;b) \mapsto \mat{G}_n \mapsto p_n \mapsto f_n(b).
\]
The differentiability of all these mappings would imply the differentiability of $f_n$, but milder conditions suffice as well. Let $\mat{G}_n(v)$ denote the discrete generator associated with~$v$. Since the number of partition elements, $n$, does not change throughout the computation, for simplicity we drop the subscript~$n$.

\medskip\noindent$\bullet$ $b\mapsto v$\\
A weaker assumption than the differentiability of $v(\cdot;b)$ w.r.t.~$b$ suffices here, see Remark~\ref{rem:zeroMeasureSet} below.

\medskip\noindent$\bullet$ $v\mapsto\mat{G}$
\begin{lemma}
Let~$v$ be a continuous and~$\delta v$ a bounded vector field from~$\set{X}$ to~$\R^d$. Further assume that ${m_{d-1}(\partial\set{X}_i\cap\{v\cdot n_i=0\})=0}$ for all~$i=1,\ldots,n$, where~$n_i$ denotes the outer normal on~$\partial\set{X}_i$. Then the directional derivative of~$\mat{G}(v)$ in the direction~$\delta v$,
\[
D\mat{G}(v)\cdot\delta v := \lim_{\ep\to 0}\frac{\mat{G}(v+\ep\delta v)-\mat{G}(v)}{\ep},
\]
is given by
\[
\renewcommand{\arraystretch}{1.5}
D\mat{G}(v)_{ij}\cdot\delta v = \left\{\begin{array}{ll}
																\frac{1}{m(\set{X}_j)}\int_{\set{X}_{ij}^+} \delta v(x)\cdot n_j(x)\,dm_{d-1}(x), & i\neq j\\
																-\frac{1}{m(\set{X}_i)}\int_{\set{X}_{ii}^+} \delta v(x)\cdot n_i(x)\,dm_{d-1}(x), & i=j.
																\end{array}\right.
\]
where $\set{X}_{ij}^+ = \partial\set{X}_i\cap\partial\set{X}_j\cap\{v\cdot n_j\ge0\}$.
\end{lemma}
\begin{proof}
We have for $i\neq j$
\begin{eqnarray*}
D\mat{G}(v)_{ij}\cdot\delta v & = & \lim_{\ep\to0}\frac{1}{\ep}\left(\frac{1}{m(\set{X}_j)}\int_{\partial\set{X}_i\cap\partial\set{X}_j}\left(\left(v+\ep\delta v\right)\cdot n_j\right)^+\,dm_{d-1} - \int_{\partial\set{X}_i\cap\partial\set{X}_j}\left(v\cdot n_j\right)^+\,dm_{d-1}\right) \\
															& = & \lim_{\ep\to0}\frac{1}{\ep}\left(\frac{1}{m(\set{X}_j)}\int_{\set{X}_{ij}^+}\left(\left(v+\ep\delta v\right)\cdot n_j\right)^+\,dm_{d-1}+o(\ep) - \int_{\set{X}_{ij}^+}v\cdot n_j\,dm_{d-1}\right) \\
															& = & \frac{1}{m(\set{X}_j)}\int_{\set{X}_{ij}^+}\delta v\cdot n_j\,dm_{d-1},
\end{eqnarray*}
where $o(\ep)$ denotes a function $g(\ep)$ such that $g(\ep)/\ep\to 0$ as $\ep\to0$. The second equation follows from the uniform continuity of $v$ on compact sets and from the boundedness of $\delta v$, the third one from the condition $m_{d-1}(\partial\set{X}_i\cap\{v\cdot n_i=0\})=0$. The proof for $i=j$ is the same.
\end{proof}
\begin{remark} 	\label{rem:zeroMeasureSet}
The previous lemma still holds if $\delta v$ is only defined on $\set{X}\setminus\set{N}$, and the set~$\set{N}$ satisfies $m_{d-1}(\partial\set{X}_i\cap\set{N})=0$ for every $i$; i.e.\ there is no fully $d-1$ dimensional intersection between the boundaries of the partition elements and the set of points where $\delta v$ is not defined. The application we have in mind is the case where $v$ is continuous but only piecewise differentiable w.r.t.\ $b$. Such a case may arise if due to some technical limitation the control has a maximal amplitude; we will show this kind of an example later on.
\end{remark}

\smallskip\noindent$\bullet$ $\mat{G}\mapsto p$\\
Next we address the differentiability of $p$ w.r.t.\ $\mat{G}$. As in Section~\ref{ssec:absProb}, let $\widehat{\mat{G}}$ denote the part of~$\mat{G}$ with the jump rates between the~$\set{X}_i$ which are not contained in~$\set{T}$. Let $q_i=\sum_{j\in\set{T}}\mat{G}_{ji}$. If $\hat{p}$ denotes the absorption probabilities corresponding to partition elements contained in ${\set{X}\setminus\set{T}}$, Proposition~\ref{prop:absprob} shows that
\[
\hat{p} = -\widehat{\mat{G}}^{-T} q.
\]
Hence $p$ is arbitrarily smooth in the entries of $\mat{G}$ and we get by elementary calculus
\begin{lemma}
The directional derivative of $\hat{p}(\mat{G})$ in the direction $\delta\mat{G}$ is given by
\[
D\hat{p}(\mat{G})\cdot\delta\mat{G} = \widehat{\mat{G}}^{-T}\widehat{\delta\mat{G}}^{T}\widehat{\mat{G}}^{-T}q-\widehat{\mat{G}}^{-T}\delta q;
\]
and $Dp_i(\mat{G})=0$ for~$i\in\set{T}$.
\end{lemma}

\medskip\noindent$\bullet$ $p\mapsto f(b)$\\
Finally, the first summand of $f(b)$ is linear in $p$, thus differentiable; and the second summand is clearly differentiable w.r.t.\ $b$.

\paragraph{Summary: optimization procedure.}
Under the assumptions made in this section the objective function $f$ is differentiable, and we sum up the optimization procedure by the gradient method as an algorithm.

\noindent\textit{Initialization:}\\
Let some partition $\{\set{X}_1,\ldots,\set{X}_n\}$ of~$\set{X}$ be given. Choose some $b_0\in\R^r$ such that $\set{T}\subsetneq\set{D}(b_0)$, a tolerance threshold $\mathrm{TOL}>0$, and step sizes $\gamma_0,\gamma_1,\ldots$. Set
\[
f(b) = \sum_{i=1}^nm(\set{X}_i)p_i - \alpha|b|^2.
\]
\textit{For $k=0,1,\ldots$ perform the following steps:}
\begin{enumerate}
	\item Compute $Df(b_k)$.
	\item STOP if $|Df(b_k)|<\mathrm{TOL}$.
	\item Set $b_{k+1} = b_k + \gamma_k Df(b_k)$.
\end{enumerate}

\begin{remark}
There are methods exploiting smoothness with higher performance than the gradient method (e.g.\ the Gau{\ss}--Newton method), however it is not clear at the first sight if the mapping $v\mapsto\mat{G}$ is differentiable more than once. Also, there are even more sophisticated first order methods than the simple gradient method~\cite{Nest83}. Using these methods for the optimization of our problem will be the subject of future work.
\end{remark}

\subsection{Minimal absorption times}	\label{ssec:optimTime}

We have already discussed at the end of Section~\ref{ssec:absProb} that the expected absorption times,~$a_n$, computed from~$\mat{G}_n$ may be inadequate for approximating~$\tau$. We also noted that~$a_{n,i}=t_{n,i}$ if~$p_{n,i}=1$, hence we expect~$t_n$ to be a good approximation of~$\tau$ on~$\set{D}$. Under the same assumptions as in Section~\ref{ssec:optim}, and by using the short-hand notation~$i\in\set{D}_0$ for~$\set{X}_i\subset\set{D}_0$, we aim to minimize the objective function
\[
f_n(b):= \sum_{i\in\set{D}_0}m(\set{X}_i) t_{n,i} + \alpha|b|^2
\]
for some prescribed region~$\set{D}_0$. Just as above, we drop the subscript~$n$ for simplicity.

The strategy is analogous to the one in the previous section: we establish the differentiability of~$f$ w.r.t.~$b$, compute the derivative~$Df(b)$, and apply the simple gradient descent method
\[
b_{k+1} = b_k - \gamma_k Df(b_k),\qquad k=0,1,\ldots
\]
to compute a local minimum of the objective function.

The differentiability of~$f$ is proven along the same lines as in the previous section, except for $Dt(\mat{G})$. Denote by $\hat{t}$ the vector with entries~$t_i$ whith~$i\in\set{X}\setminus\set{T}$.
\begin{lemma}
The directional derivative of $\hat{t}(\mat{G})$ in the direction $\delta\mat{G}$ is given by
\[
D\hat{t}(\mat{G})\cdot\delta\mat{G} = \widehat{\mat{G}}^{-T}\widehat{\delta\mat{G}}^{T}\widehat{\mat{G}}^{-T}e,
\]
where~$e=(1,\ldots,1)^T$.\\
Further holds $Dt_i(\mat{G})=0$ for~$i\in\set{T}$.
\end{lemma}
\begin{proof}
We have from~\eqref{eq:termtimes} that $\hat{t} = -\widehat{\mat{G}}^{-T}e$. Differentiation w.r.t.\ $\mat{G}$ yields the first claim. The second follows from $t_i(\mat{G})\equiv 0$ for~$i\in\set{T}$.
\end{proof}

Note that unless there is a~$b$ such that~${\set{D}_0\subseteq\set{D}(b)}$, we cannot expect the objective function to obtain finite values, since $\tau(x;b)=\infty$ for~$x\notin\set{D}(b)$. To exclude such a case we assume that we already start the optimization with a~$b_0$ such that~$\set{D}_0\subseteq\set{D}(b_0)$.\footnote{We may find such a $b_0$ by taking the maximization problem from Section~\ref{ssec:optim} with the objective function $f(b)=\int_{\set{D}_0}\chi_{\set{D}(b)}-\alpha|b|^2$ for some very small $\alpha>0$, with its discrete counterpart $f(b)=\sum_{i\in\set{D}_0}m(\set{X}_i)p_i(b)-\alpha|b|^2$.} Further, we have to assure that none of the $b_k$ is such that~$\set{D}_0\nsubseteq\set{D}(b_k)$. Observe that~$\set{D}_0\subseteq\set{D}(b_k)$ is equivalent with
\[
g(b):= \sum_{i\in\set{D}_0}m(\set{X}_i)p_i = \sum_{i\in\set{D}_0}m(\set{X}_i),
\]
if $\set{D}_0$ is the union of partition elements; what we assume from now on. Thus, we have to assure that the sequence $\{g(b_k)\}_{k=0,1,\ldots}$ is constant, in particular non-decreasing. If the increment $\Delta b_k := Df(b_k)$ in the iteration of the gradient method is small enough, we have
\[
g(b_{k+1})-g(b_k) \approx Dg(b_k)\cdot\Delta b_k,
\]
and the condition $Dg(b_k)\cdot\Delta b_k \ge 0$ assures that the above sequence is essentially non-decreasing. Note that the computation of~$Dg$ can be done similarly to that of~$Df$ in Section~\ref{ssec:optim}. So if ${Dg(b_k)\cdot\Delta b_k}<0$, we use the projection of~$\Delta b_k$ onto~$Dg(b_k)^{\perp}:=\{x\in\R^r\,\vert\, x\cdot Dg(b_k)=0\}$ as increment.

\paragraph{Summary: optimization procedure.}
Under the assumptions made in this section the objective function $f$ is differentiable, and we sum up the optimization procedure by the gradient descent method as an algorithm.

\noindent\textit{Initialization:}\\
Let some partition $\{\set{X}_1,\ldots,\set{X}_n\}$ of~$\set{X}$, and some~$\set{D}_0\supset\set{T}$ be given. Choose some $b_0\in\R^r$ such that $\set{D}_0\subsetneq\set{D}(b_0)$, and a tolerance threshold $\mathrm{TOL}>0$. Set
\[
f(b) = \sum_{i\in\set{D}_0}m(\set{X}_i)t_i + \alpha|b|^2.
\]
\textit{For $k=0,1,\ldots$ perform the following steps:}
\begin{enumerate}
	\item Compute $\Delta b_k:=Df(b_k)$ and $Dg(b_k)$.
	\item STOP if $|\Delta b_k|<\mathrm{TOL}$.
	\item IF $Dg(b_k)\cdot \Delta b_k<0$, set \[ \Delta b_k := \Delta b_k - \frac{\Delta b_k\cdot Dg(b_k)}{|Dg(b_k)|^2}Dg(b_k).\]
	\item Set $b_{k+1} = b_k - \Delta b_k$.
\end{enumerate}

\begin{remark}
There are more established methods for constrained optimization than the one we use here. Their application, however, is beyond the scope of this work.
\end{remark}

\subsection{Parameter-affine systems}	\label{ssec:linear}

Here we consider the model problems for the case where $v(\cdot;b)$ is affine-linear in $b$; i.e.
\[
v(x;b) = v_0(x) + v_c(x;b),
\]
with $v_c(\cdot;b)$ being linear in $b$. Take for example
\[
v(x;b) = v_0(x) + v_c(x)\mat{B}x,
\]
where $v_0$ and the columns of $v_c$ are vector fields from $\set{X}$ to $\R^d$, and $\mat{B}$ is obtained from $b$ by reshaping the vector to a matrix.

For this class of vector fields an additional approximation step may save a huge amount of computational efforts. Let $e_i\in\R^r$, $i=1,\ldots,r$, be such that $e_{i,i}=1$ and $e_{i,j}=0$ for $j\neq i$. Then, by linearity,
\[
v(x;b) = v_u(x) + \sum_{i=1}^rb_i v_c(x;e_i),
\]
and we use
\begin{equation}
\mat{G}_n^*(v):=\mat{G}_n(v_0)+\sum_{i=1}^r |b_i|\mat{G}_n\left(\mathrm{sign}(b_i)v_c(\cdot;e_i)\right)
\label{eq:discretizationLinear}
\end{equation}
instead of $\mat{G}_n(v)$ in our computations.\footnote{Note that in general $\mat{G}_n(v)$ is not linear in $v$, but it approximates the dynamics in a distributional sense, cf.\ above, hence we expect it to behave linearly ``in the limit'', i.e.\ if $n$ is large enough. Any further discussion on the topic would lead beyond the scope of this paper, involving semigroups of transfer operators associated with flows; we refer the reader to~\cite{FroJK10}.} The reason why $\mat{G}_n^*(v)$ is not simply a linear combination is the fact that linear combination of generator matrices does not have to be a generator matrix. Conical combination of generator matrices, however, is a generator matrix. Unfortunately, differentiability of~$\mat{G}_n^*(v)$ w.r.t.~$b_i$ at~${b_i=0}$ is not guaranteed any more.

The usage of $\mat{G}_n^*(v)$ has the advantage that the computationally expensive steps of computing the generator matrix for a vector field have to be done only once, at the beginning $(2r+1)$ times, and not for all iterates~$b_k$ in the gradient method. This brings a massive speed-up in runtime at the expense of small loss in accuracy.

\section{Numerical examples} \label{sec:numexp}

\subsection{Example: Domain of attraction maximization}

Consider the two dimensional dynamical system given by
\begin{equation}
\dot{\begin{pmatrix}	x_1 \\ x_2 \end{pmatrix}}
= \underbrace{3(x_1^2+x_2^2)
\begin{pmatrix}
	x_1+2x_2+3x_2^2-50x_2^4 \\
	2x_1+3x_1^2+x_2
\end{pmatrix}}_{=v_u(x)}
+
\underbrace{\begin{pmatrix}
	-1 & 0\\
	0 & -1-2x_2
\end{pmatrix}
\mat{B}\begin{pmatrix} x_1\\ x_2 \end{pmatrix}}_{=v_c(x;b)},
\label{eq:systemE}
\end{equation}
with saturation condition ${\|v_c(x;b)\|_{\infty}\le 0.3}$. Here $b$ is the parameter vector containing the entries of the parameter matrix ${\mat{B}\in\R^{2\times 2}}$. We chose the state space to be ${\set{X}=[-1,1]\times[-1,1]}$, and the target region to be ${\set{T}=[-0.05,0.05]\times[-0.05,0.05]}$.

The goal is to maximize ${\set{D}(b)-\alpha\|b\|_2^2}$, with $\alpha=0.02$. For this we use the method described in Section~\ref{ssec:optim} (see Remark~\ref{rem:zeroMeasureSet} on how to handle the above saturation condition) with step sizes $\gamma_k=3$. Note that the linearization of system~\eqref{eq:systemE} around the origin yields $\dot \xi=-\mat{B}\xi$. We start the iteration with $\mat{B}_0=\begin{pmatrix}1 & 1\\ 0 & 1\end{pmatrix}$, and do 15 gradient steps.

In Figure~\ref{fig:DOAs} we show the absorption probabilities corresponding to the system with parameter matrix $\mat{B}_{15}$ (for the one computed with the gradient method for the particular partition, respectively), from left to right for a $64\times 64$, a $128\times 128$ and a $256\times 256$ uniform partition of the state space.
\begin{figure}[htb]
	\centering	\includegraphics[width=0.325\textwidth]{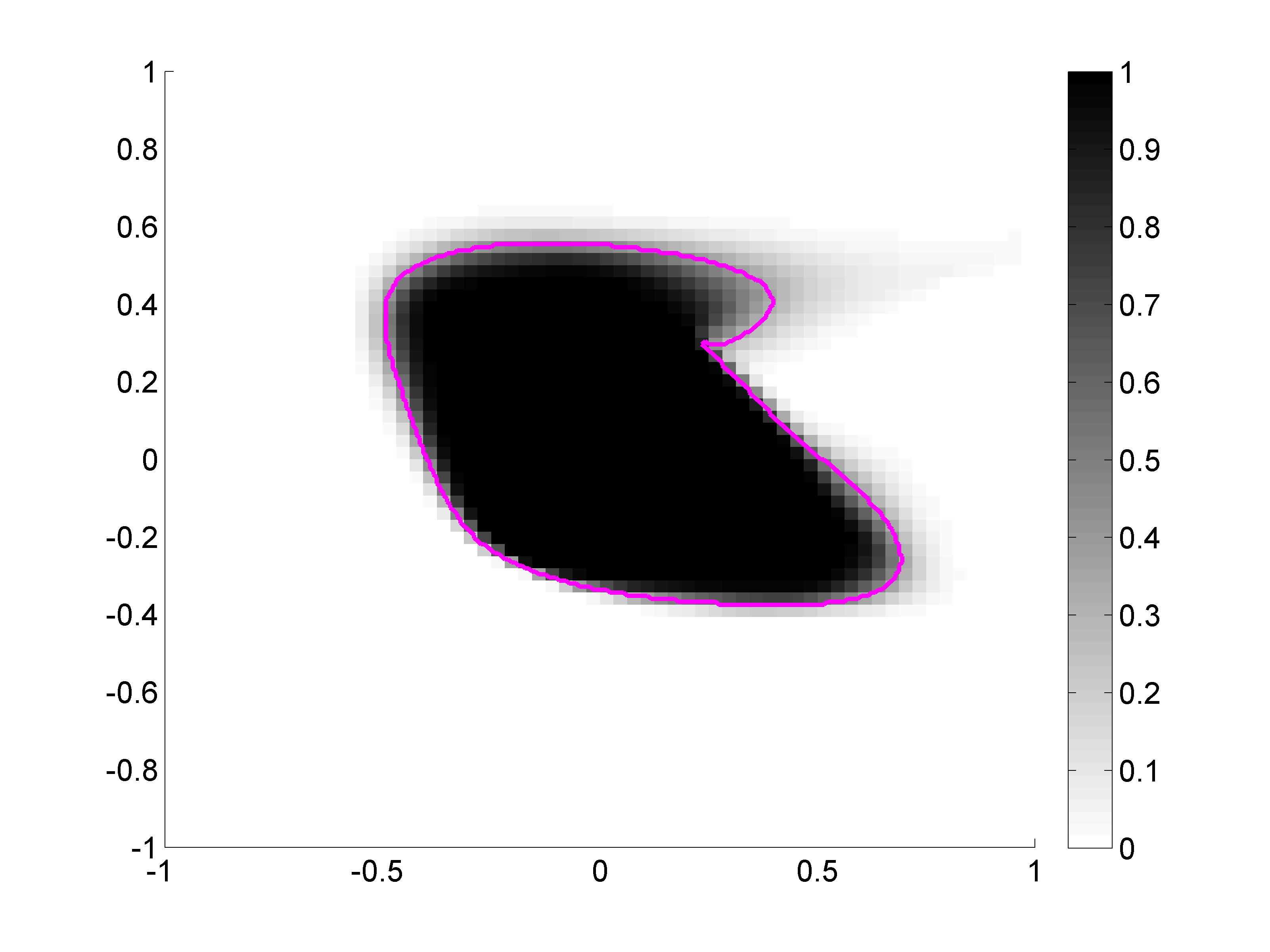} \hfill
	\includegraphics[width=0.325\textwidth]{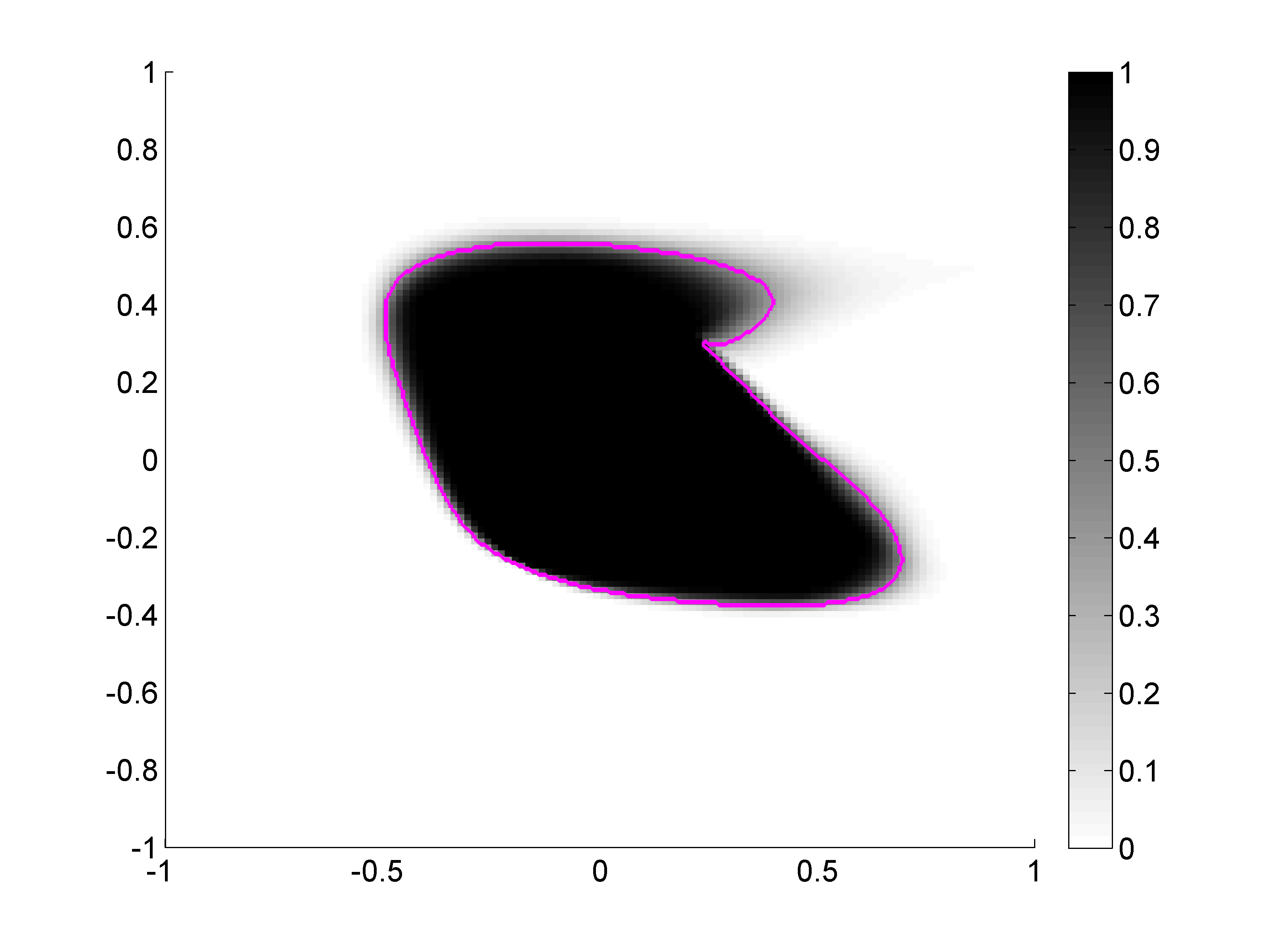} \hfill
	\includegraphics[width=0.325\textwidth]{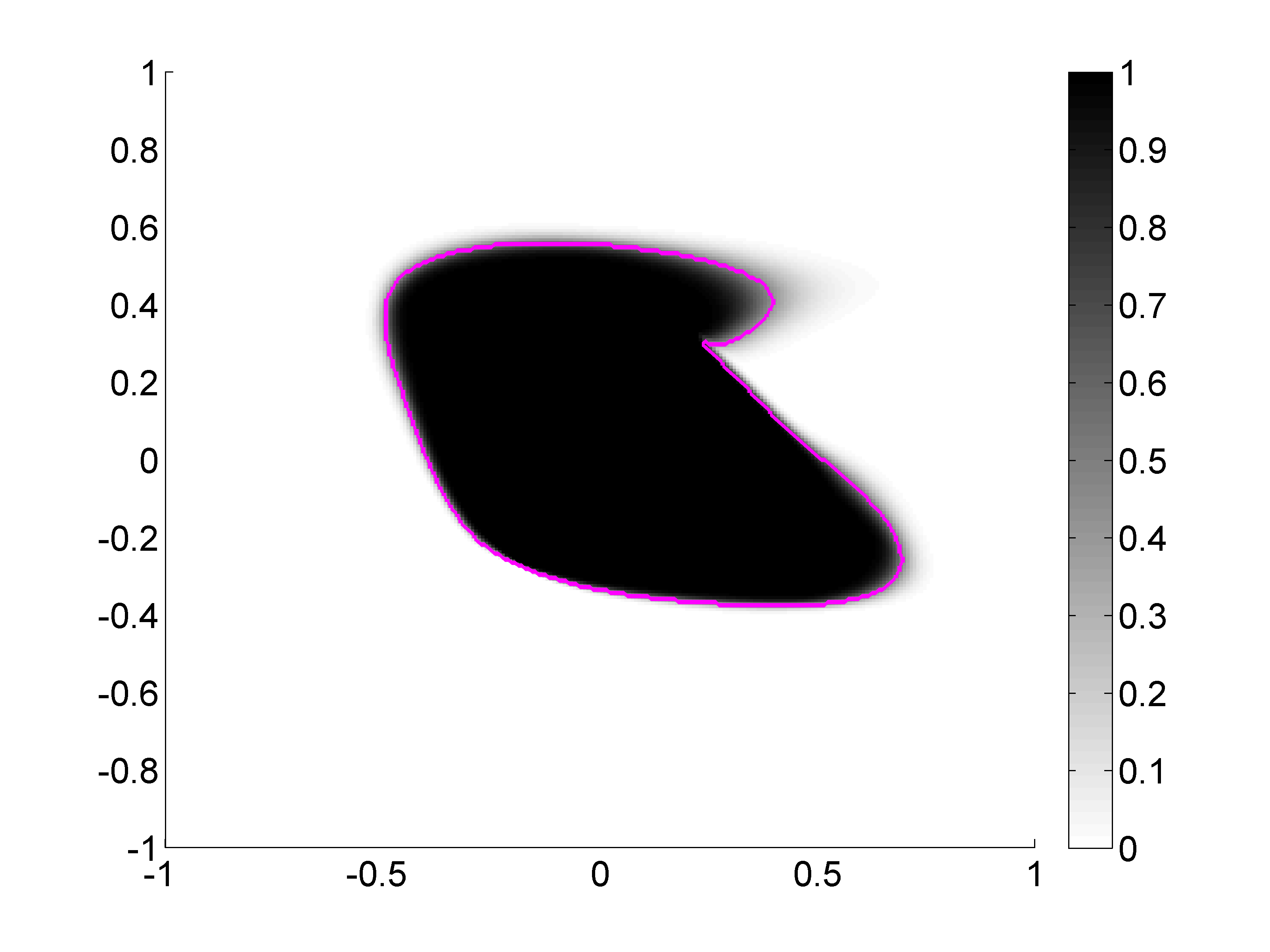}
	\caption{The absorption probabilities corresponding to the parameter $b_{15}$ for a $64\times 64$ (left), a $128\times 128$ (center) and a $256\times 256$ (right) uniform partition of the state space. The contour indicates the border of the DOA computed by direct simulation.}
	\label{fig:DOAs}
\end{figure}

Table~\ref{tab:objectivesDOA} shows the change in the objective function after 15 gradient steps for the different partitions. For all three discretizations ${\|Df(b_{15})\|_2 < 10^{-3}}$. The objective function increased by 15\% compared with the naive initial parameter array computed from linearization.
\begin{table}[htbp]
	\centering
	\renewcommand{\arraystretch}{1.5}
		\begin{tabular}{c|ccc}
			& $64\times64$ & $128\times 128$ & $256\times256$ \\ \hline
		$f(b_0)$ & 0.5888 & 0.5970 & 0.6026 \\
		$f(b_{15})$ & 0.6768 & 0.6862 & 0.6924
		\end{tabular}
	\caption{Initial and final values of the objective function for the three different partitions.}
	\label{tab:objectivesDOA}
\end{table}

There is no reason to expect the objective function to have only one local minimum (which will be the global one as well). Thus we applied the algorithm to different initial parameter values. The resulting sequence, however, converged always to the same optimum.

To save computational efforts one could start an optimization with a coarse resolution, and successively change to a finer partition at some appropriate iteration step.

\subsection{Example: Absorption time minimization}

Consider system~\eqref{eq:systemE} with the same saturation condition as above. Let ${\set{X}=[-1,1]\times[-1,1]}$, \linebreak[4] ${\set{T}=[-0.03,0.03]\times[-0.03,0,03]}$ and ${\set{D}_0=\left\{x\in\R^2\,\big\vert\,|x|\le 0.3\right\}}$.

The goal is to minimize
\[
f(b) = \int_{\set{D}_0}\tau(x)\,dx + \alpha |b|^2,
\]
with $\alpha=0.02$. For this we use the method described in Section~\ref{ssec:optimTime} with step sizes $\gamma_k=3$, and 15 steps. We start the iteration at the optimal parameter value computed for the maximal DOA above, which is~${\mat{B}_0=\begin{pmatrix}0.89 & 0.35\\ 0.75 & 1.4\end{pmatrix}}$.

In Figure~\ref{fig:times} we show the optimal absorption times computed with the gradient descent method for three different partitions. In every case the optimum seems to occur at the boundary of the feasible region~${\left\{b\,\big\vert\, \set{D}_0\subset \set{D}(b)\right\}}$. This matches well with the fact that the boundary of~$\set{D}_0$ (indicated by the circular contour in the figure) is very close to the boundary of the set with high absorption probabilities (the plotted boxes indicate absorption probabilities greater than $0.9$).
\begin{figure}[htb]
	\centering	\includegraphics[width=0.325\textwidth]{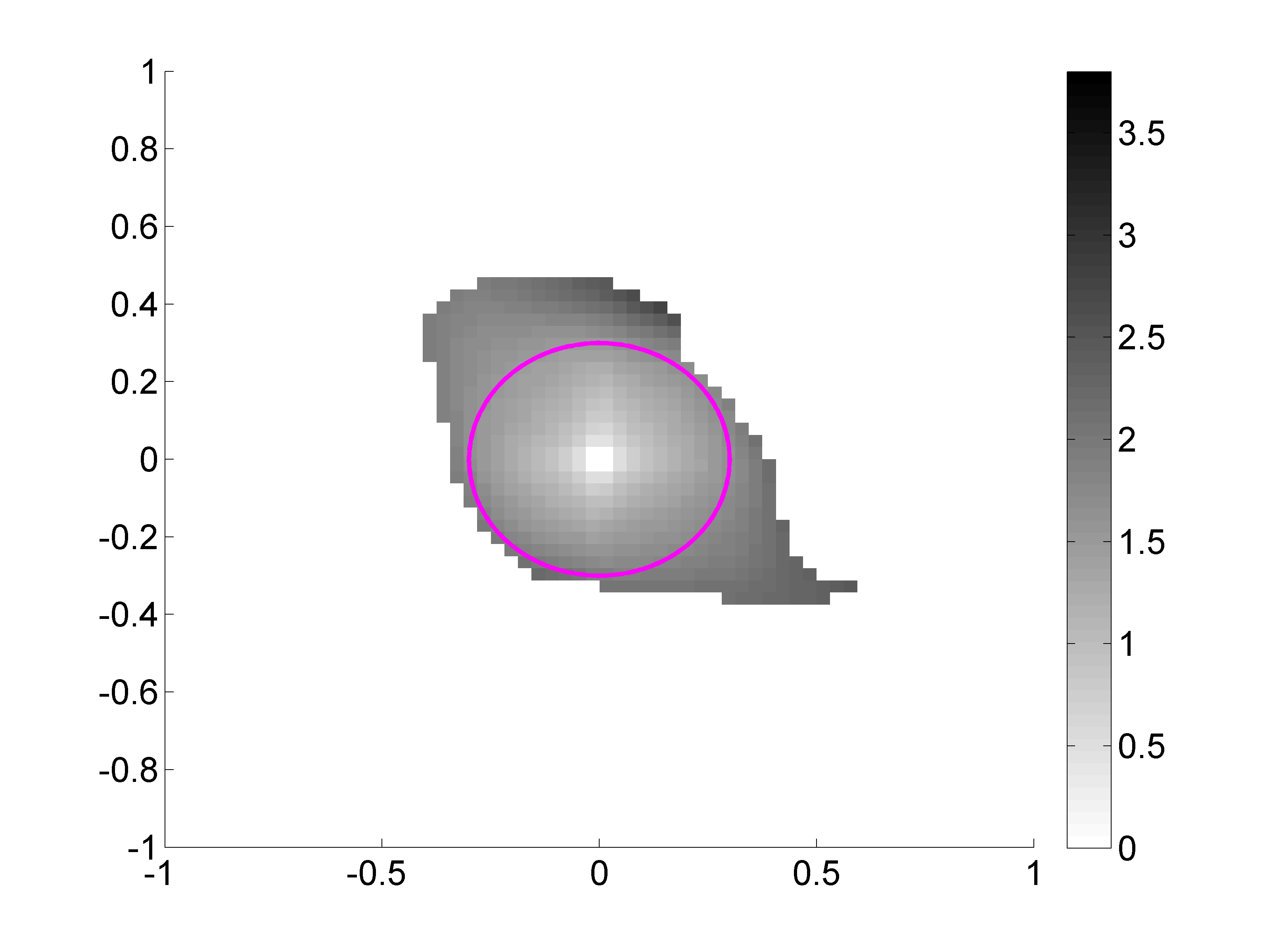} \hfill	\includegraphics[width=0.325\textwidth]{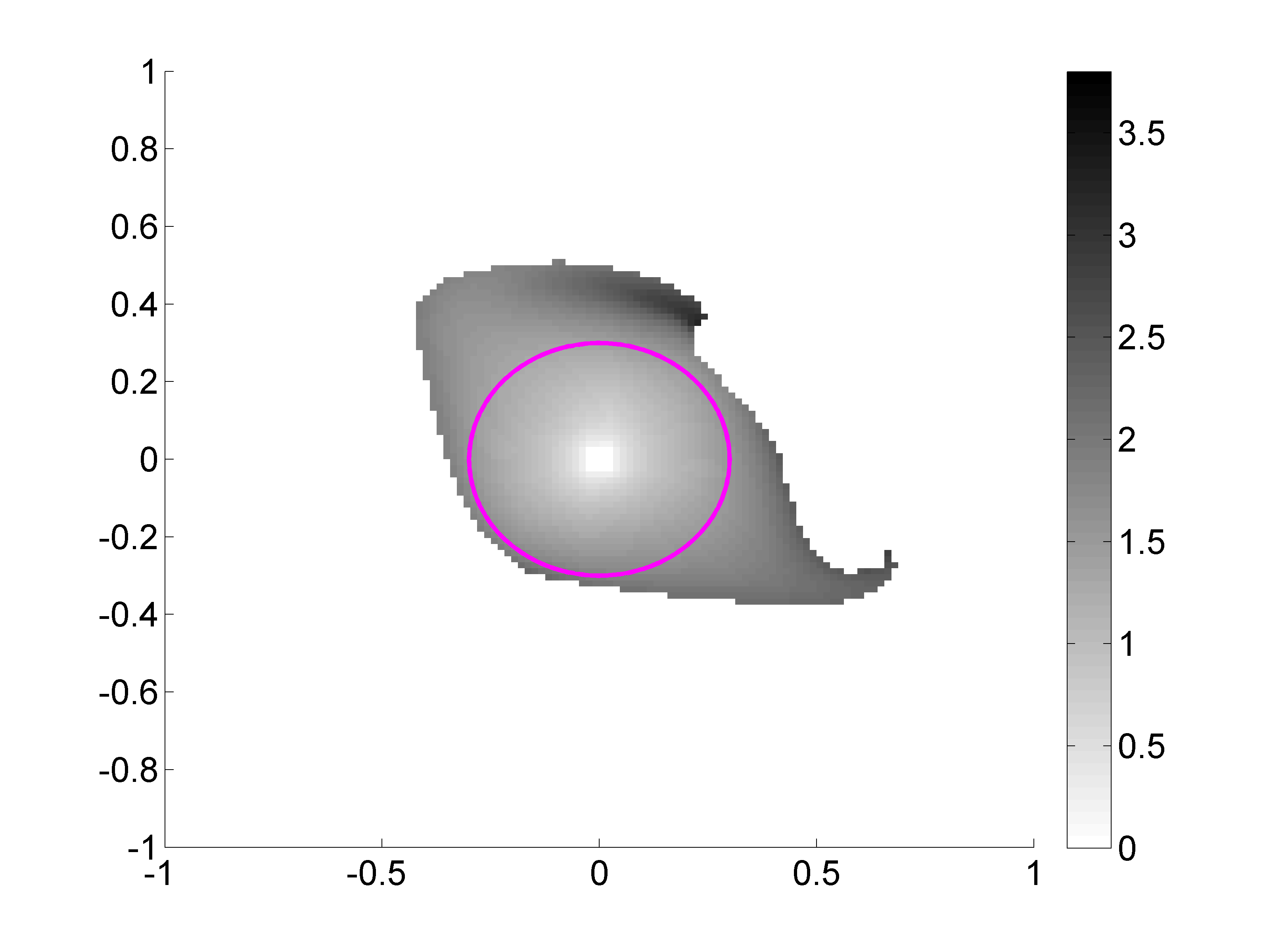} \hfill	\includegraphics[width=0.325\textwidth]{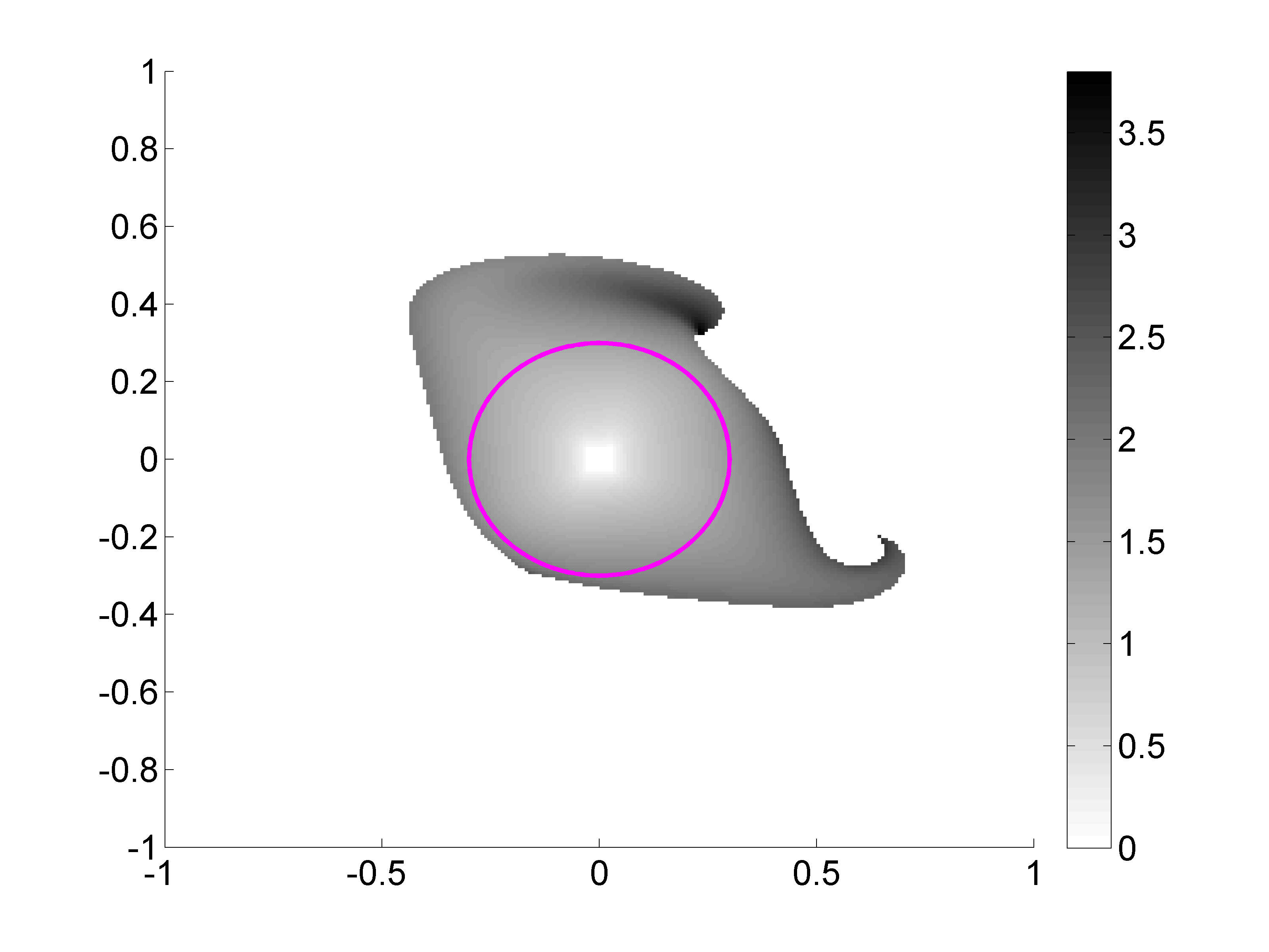}
	\caption{The absorption times corresponding to the parameter $b_{15}$ of the discrete generator computed on a $64\times 64$ (left), a $128\times 128$ (center) and a $256\times 256$ (right) uniform partition of the state space. Only those boxes have been plotted which have an absorption probability more than $0.9$. The contour indicates the region~$\set{D}_0$.}
	\label{fig:times}
\end{figure}

After 15 steps the gradients (more precisely, the gradients projected to $Dg(b_{k})^{\perp}$, since the iteration converges to the boundary of the feasible set; see Section~\ref{ssec:optimTime}) were for all discretizations smaller than $4\cdot 10^{-3}$. Table~\ref{tab:objectivesTime} shows the change in the objective function after 15 gradient steps for the different partitions.
\begin{table}[htb]
	\centering
	\renewcommand{\arraystretch}{1.5}
		\begin{tabular}{c|ccc}
			& $64\times64$ & $128\times 128$ & $256\times256$ \\ \hline
		$f(b_0)$ & 1.620 & 1.142 & 0.9418 \\
		$f(b_{15})$ & 0.5278 & 0.4750 & 0.4436
		\end{tabular}
	\caption{Initial and final values of the objective function for the three different partitions.}
	\label{tab:objectivesTime}
\end{table}
It is worth to note that for different starting points~$b_0$ the iteration may run into different (but very similar) local minima which all lie at the boundary of the feasible region~${\left\{b\,\big\vert\, \set{D}_0\subset \set{D}(b)\right\}}$. Also, for finer partitions these local minima are closer to each other. This suggests that the occurrence of multiple minima is only due to the finite discretization of the problem, and they all collapse to one minimum as the diameter of the partition elements tend to~0.

\subsection{Example: Affine parameter dependence}

We test the modification introduced in Section~\ref{ssec:linear} on the system
\begin{equation}
\dot{\begin{pmatrix}	x_1 \\ x_2 \end{pmatrix}}
= \underbrace{3(x_1^2+x_2^2)
\begin{pmatrix}
	x_1+2x_2+3x_2^2-50x_2^4 \\
	2x_1+3x_1^2+x_2
\end{pmatrix}}_{=v_u(x)}
+
\underbrace{\begin{pmatrix}
	-1 & 0\\
	0 & -0.1
\end{pmatrix}
\mat{B}\begin{pmatrix} x_1\\ x_2 \end{pmatrix}}_{=v_c(x;b)},
\label{eq:systemEmod}
\end{equation}
which has the desired affine dependence on the parameter array~$b$.

First we compute the discrete generators corresponding to the system~\eqref{eq:systemEmod} with parameter \linebreak[4] ${\mat{B}_0 = \begin{pmatrix}0.1 & 10\\0 & 15\end{pmatrix}}$ by the direct discretization~\eqref{eq:disc_gener} and by the modified discretization~\eqref{eq:discretizationLinear}, respectively. The state space is $[-1,1]\times[-1,1]$, and we apply a $40\times40$ uniform partition in both cases. Next, we compute the absorption probabilities for the target~${\set{T}=[-0.05,0.05]\times[-0.05,0.05]}$ using the one and then the other discretization. Figure~\ref{fig:LinVsStandard1} shows the results.
\begin{figure}[htb]
	\centering	\includegraphics[width=0.35\textwidth]{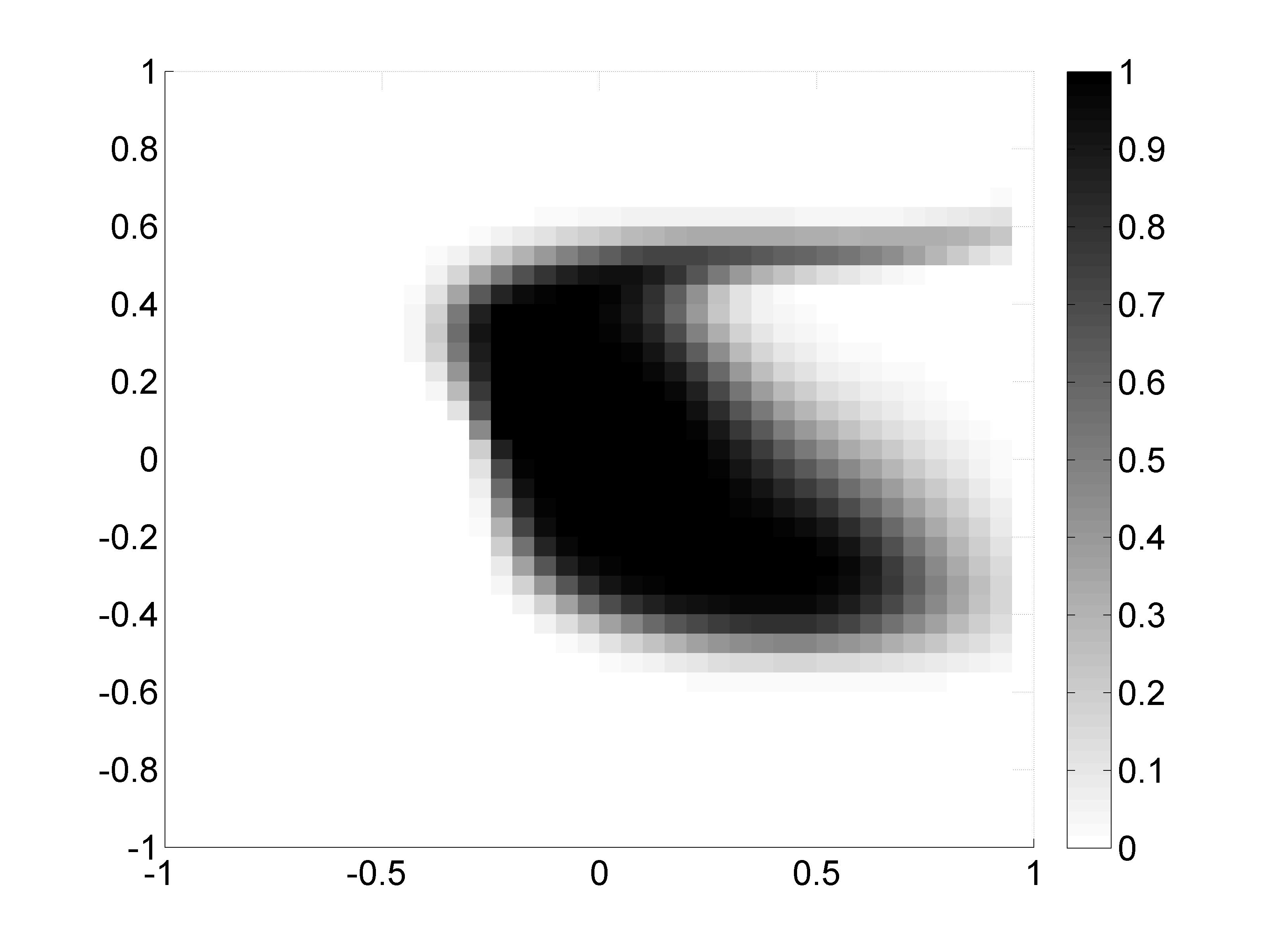}
\includegraphics[width=0.35\textwidth]{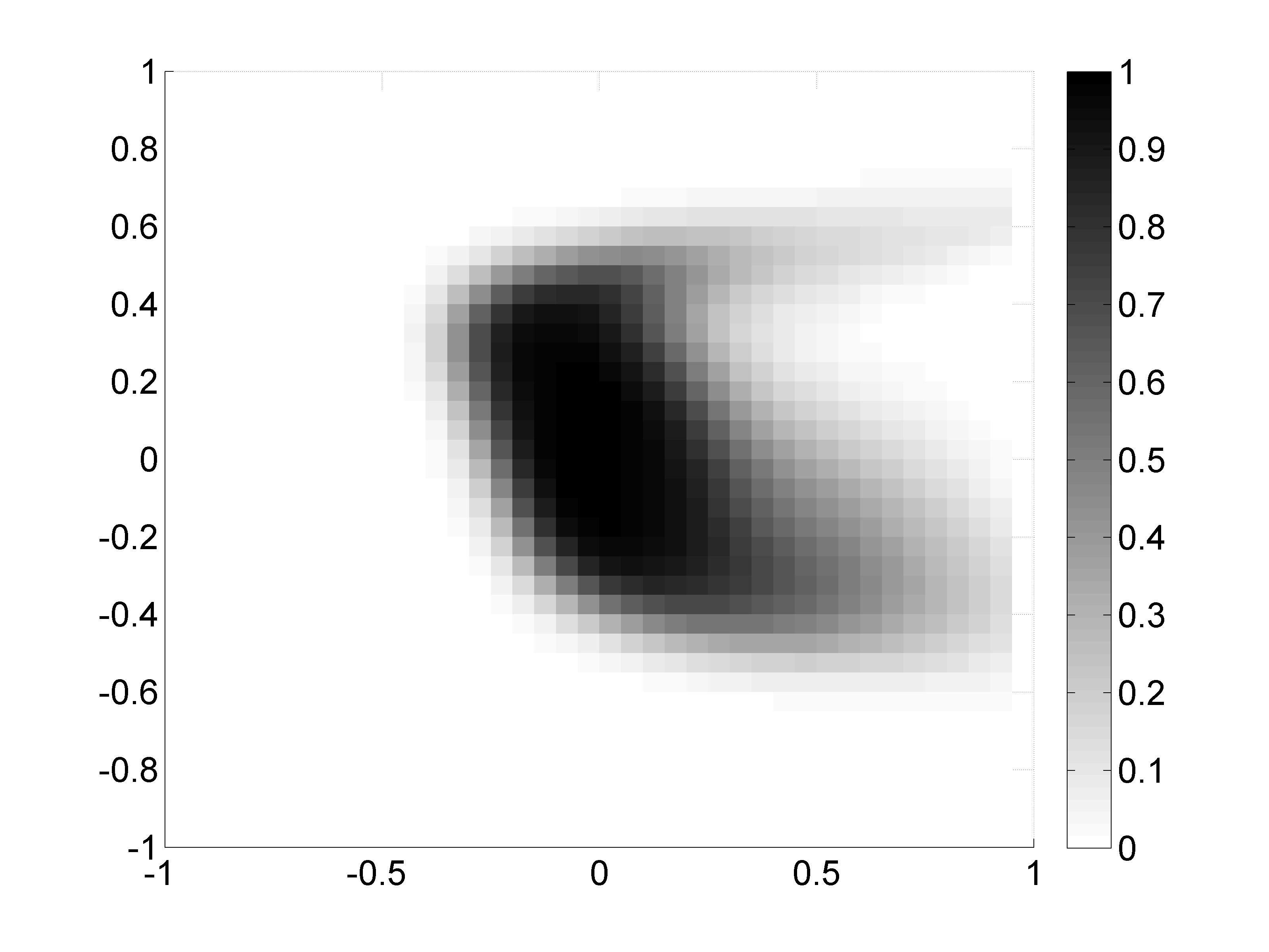}
	\caption{Absorption probabilities computed using the standard discretization (left) and the modified discretization~\eqref{eq:discretizationLinear} (right). The larger diffusivity of the modified discretization is reflected in the milder descent of absorption probabilities.}
	\label{fig:LinVsStandard1}
\end{figure}
To understand why does the modified discretization give more ``blurred'' absorption probabilities, note the following. The standard discretization takes a linear combination of some vector fields and computes the discrete generator from it. The modified method takes discrete generators of some vector fields and combines them linearly to one generator. This means for one element in the partition, that if the component vector fields~$v(\cdot,e_i)$ show locally in many different directions, the combined generator of the modified discretization will yield transition rates in many different neighboring partition elements. Hence, the generator from the modified discretization may have a larger ``diffusivity'', forcing absorption probabilities of neighboring boxes to be closer to each other. This is what we observe as blurring in the figure.

This diffusivity effect decreases as we use finer partitions. Applying 15 steps of the gradient method for the problem of maximizing ${\set{D}(b)-0.02|b|^2}$ with a ${256\times 256}$ partition and $\mat{B}_0$ as above, the difference in the objective function computed by the two discretizations is only $3.5\%$, and the corresponding absorption probabilities are shown in Figure~\ref{fig:LinVsStandard2}. Note that the modified method computes 9 generators at the beginning, and then no more for the whole iteration. Since in this example approximately 45 steps are needed that the gradient of the objective function falls under $10^{-3}$ (the step sizes are constant,~$\gamma_k=3$), the modified method has a notable advantage over the standard one.
\begin{figure}[htb]
	\centering \includegraphics[width=0.35\textwidth]{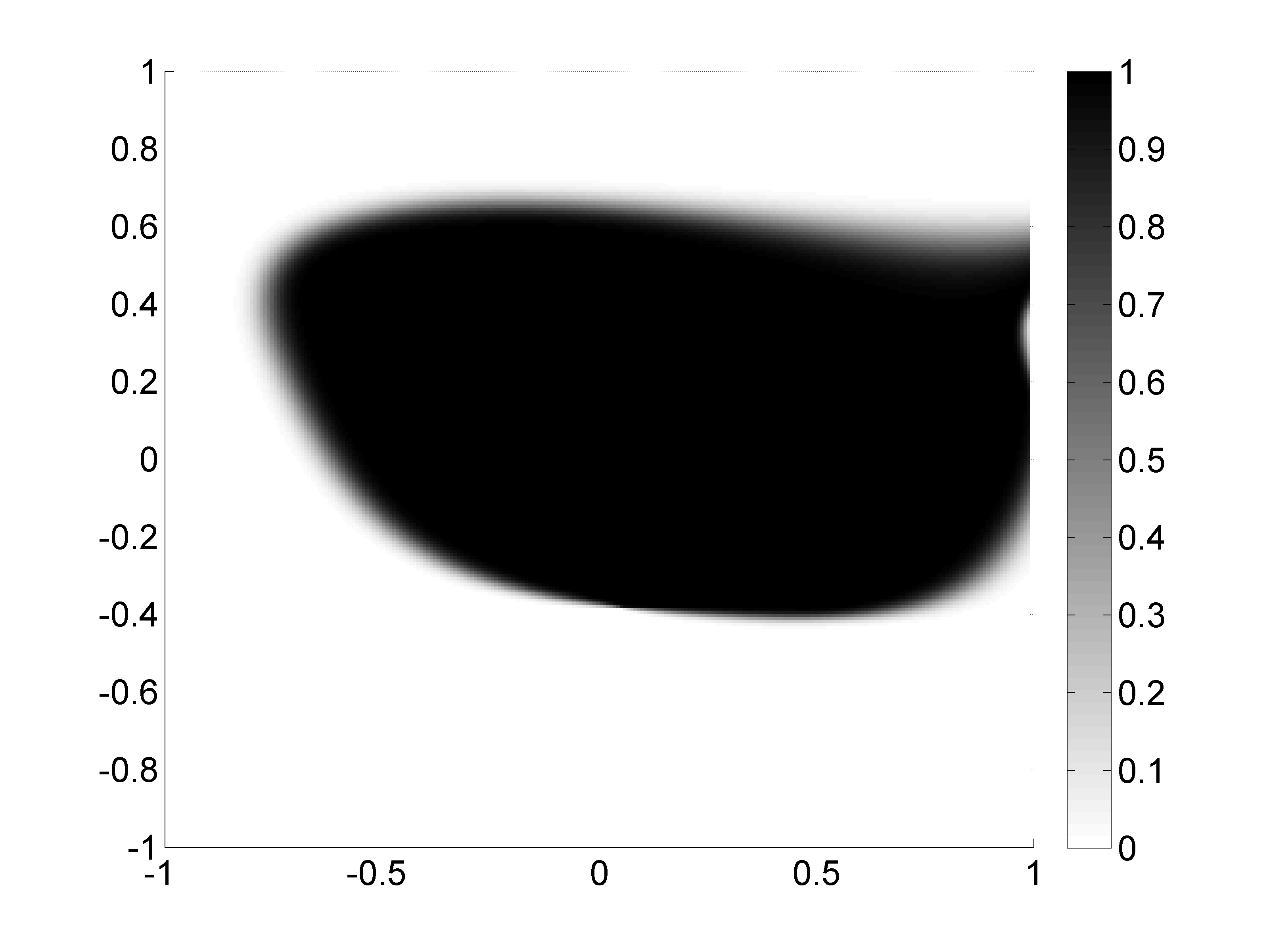}
\includegraphics[width=0.35\textwidth]{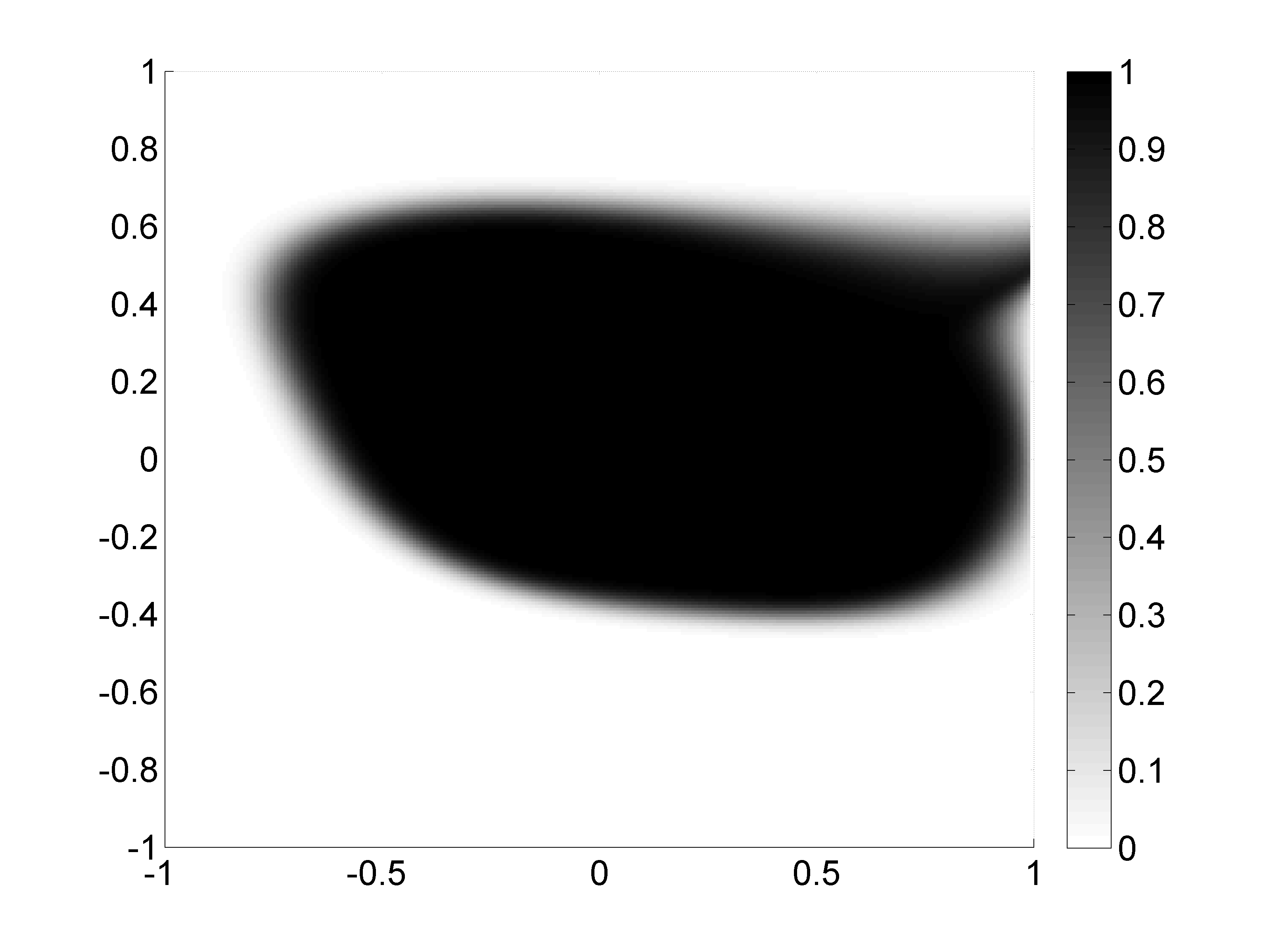}
	\caption{Absorption probabilities for system~\eqref{eq:systemEmod} with~$b_{15}$ computed from the standard discrete generator~$\mat{G}(v)$ (left) and  from the modified discrete generator~$\mat{G}^*(v)$ (right) on a $256\times256$ partition of $[-1,1]\times[-1,1]$. The corresponding objective functions differ only up to~$3.5\%$.}
	\label{fig:LinVsStandard2}
\end{figure}

\begin{remark}
If the state space is partitioned into $N$ elements (and if there are $\mathcal{O}(N)$ of them not belonging to the target), then the computation of the discrete generator has $\mathcal{O}(N)$ computational complexity, while the computation of absorption probabilities and times has $\mathcal{O}(N^3)$ (due to the solution of a system of linear equations by LU-decomposition). For the examples presented here the computation of the discrete generator was always by far the computationally most expensive step, however for finer partitions, and especially in higher system dimensions, the computational costs of the solution of the linear equations may dominate. These computations can be carried out more quickly by hierarchical GMRES techniques (see~\cite{KPPHD}, Section~5.7.4), so that the modified discretization for the case of linear dependence on the parameters still induces a considerable speed-up in the optimization against the standard discretization~\eqref{eq:disc_gener}.
\end{remark}

\section{Conclusions} \label{sec:conclusions}

We have proposed a method for the optimization of the global stability properties of time-continuous autonomous parameter-dependent systems. It uses an approximation of the deterministic dynamics by a Markov jump process (essentially the spatial discretization of the upwind method). The main computational properties of the method are that no trajectory simulation is needed, and that the computation of the objective function and of its derivative is simple and consist mostly of explicit steps. If the system is affine-linear in the parameters, a considerable speed-up is achieved by a slight modification.

A quantitative performance analysis of the method, in particular convergence statements, are subject of future work, just as the incorporation of adaptivity. The advantages of combining the basic idea with some more sophisticated optimization approaches are also to be investigated.


\section*{Acknowledgments}
The authors would like to thank Oliver Junge and Gary Froyland for helpful discussions.

\bibliography{KoVo11}
\bibliographystyle{alpha}

\end{document}